\documentclass{amsart}
\usepackage{amsmath}
\usepackage{amssymb}
\usepackage{mathrsfs,comment}
\usepackage{mathtools}
\usepackage[usenames,dvipsnames]{color}
\usepackage[normalem]{ulem}
\usepackage{url}
\usepackage[all,arc,2cell]{xy}
\UseAllTwocells
\usepackage{enumerate}
\usepackage{color}
\usepackage[perpage]{footmisc}
\usepackage{tikz}
\usepackage{tikz-cd}
\usepackage{hyperref}  
\hypersetup{%
  bookmarksnumbered=true,%
  colorlinks=true,%
  linkcolor=blue,%
  citecolor=blue,%
  filecolor=blue,%
  menucolor=blue,%
  urlcolor=blue,%
  pdfnewwindow=true,%
  pdfstartview=FitBH
  }

\def\makeautorefname#1#2{\expandafter\def\csname#1autorefname\endcsname{#2}}
\def\equationautorefname~#1\null{(#1)\null}
\makeautorefname{footnote}{footnote}%
\makeautorefname{item}{item}%
\makeautorefname{figure}{Figure}%
\makeautorefname{table}{Table}%
\makeautorefname{part}{Part}%
\makeautorefname{appendix}{Appendix}%
\makeautorefname{chapter}{Chapter}%
\makeautorefname{section}{Section}%
\makeautorefname{subsection}{Section}%
\makeautorefname{subsubsection}{Section}%
\makeautorefname{theorem}{Theorem}%
\makeautorefname{thm}{Theorem}%
\makeautorefname{cor}{Corollary}%
\makeautorefname{lem}{Lemma}%
\makeautorefname{prop}{Proposition}%
\makeautorefname{pro}{Property}
\makeautorefname{conj}{Conjecture}%
\makeautorefname{defn}{Definition}%
\makeautorefname{notn}{Notation}
\makeautorefname{notns}{Notations}
\makeautorefname{rem}{Remark}%
\makeautorefname{quest}{Question}%
\makeautorefname{exmp}{Example}%
\makeautorefname{ax}{Axiom}%
\makeautorefname{claim}{Claim}%
\makeautorefname{ass}{Assumption}%
\makeautorefname{asss}{Assumptions}%
\makeautorefname{con}{Construction}%
\makeautorefname{prob}{Problem}%
\makeautorefname{warn}{Warning}%
\makeautorefname{obs}{Observation}%
\makeautorefname{conv}{Convention}%

\newtheorem{thm}{Theorem}[section]
\newtheorem{cor}{Corollary}[section]
\newtheorem{prop}{Proposition}[section]
\newtheorem{lem}{Lemma}[section]

\theoremstyle{definition}

\makeatletter
\let\c@obs=\c@thm
\let\c@cor=\c@thm
\let\c@prop=\c@thm
\let\c@lem=\c@thm
\let\c@prob=\c@thm
\let\c@con=\c@thm
\let\c@conj=\c@thm
\let\c@defn=\c@thm
\let\c@notn=\c@thm
\let\c@notns=\c@thm
\let\c@exmp=\c@thm
\let\c@ax=\c@thm
\let\c@pro=\c@thm
\let\c@ass=\c@thm
\let\c@warn=\c@thm
\let\c@rem=\c@thm
\let\c@sch=\c@thm
\numberwithin{equation}{section}
\makeatother

\newcommand{\bbR}{\mathbb{R}}

\newcommand{\vp}{\varphi}
\renewcommand{\phi}{\varphi}

\DeclarePairedDelimiter{\abs}{\lvert}{\rvert}

\newcommand{\del}[1]{\tfrac{\partial}{\partial #1} }

\newcommand{\R}{\mathbb{R}}

\newcommand{\Z}{\mathbb{Z}}

\newcommand{\ti}[1]{\widetilde{#1}}
\newcommand{\p}{\partial}

\newcommand{\reg}{\mathrm{reg}}
\newcommand{\sing}{\mathrm{sing}}

\renewcommand{\phi}{\varphi}

\newenvironment{spm}
    {\left( \begin{smallmatrix}
    }
    { 
     \end{smallmatrix} \right)
    }

\title[Singularities of a Monge--Amp\`ere equation]{Singularities of the solution to a Monge--Amp\`ere equation on the boundary of the 3-simplex}

\author[M. Jonsson]{Mattias Jonsson}
\address{Dept. of Mathematics\\
  University of Michigan\\
  Ann Arbor, MI 48109-1043}
\email{mattiasj@umich.edu}

\author[N. McCleerey]{Nicholas McCleerey}
\address{Dept. of Mathematics\\
  Purdue University\\
  West Lafayette, IN 47907}
\email{nmccleer@purdue.edu}

\author[N. Patram]{Neil Patram}
\address{Dept. of Mathematics\\
  Georgia Institute of Technology\\
  Atlanta, GA 30044}
\email{npatram3@gatech.edu}

\author[B.W. Scott]{Benjamin W. Scott}
\address{Dept. of Mathematics\\
  University of Chicago\\
  Chicago, IL 60637}
\email{scott2002@uchicago.edu}

\begin{document}

\begin{abstract}
We show that the metric defined by the solution to the tropical Monge-Amp\`ere equation, as defined by Hultgren, Mazzon, and the first two authors, on the boundary of the 3-simplex is asymptotic to the Gross-Wilson metric on $S^2$ near each of the 6 singular points. We deduce in addition that the solution is not $C^{1,1}$ across the singular points. Compared to previous works, our starting point is the real Monge-Amp\`ere equation, as opposed to the complex structure.
\end{abstract}

\maketitle


\section{Introduction} 
\label{section: intro}

Motivated by the Strominger-Yau-Zaslow conjecture \cite{SYZ96} in mirror symmetry, in \cite{GW00}, Gross and Wilson study the ``large complex-structure limit" of generic K3 surfaces. We briefly recall their results. A generic K3 is an elliptic fibration over $\mathbb{CP}^1$ with 24 isolated singularities; after a hyperk\"ahler rotation and a renormalization, the large complex-structure limit can be reduced to the limit when the volume of the fibres tends to zero.

When the fibres are sufficiently small, Gross-Wilson glue model metrics (due to Ooguri-Vafa) near the singular fibres onto an explicit semi-flat metric, producing a nearly Ricci-flat metric on the K3. The metric is a good approximation for the genuine Ricci-flat metrics produced by the Calabi-Yau theorem, and so it is shown that the two converge to the same affine metric on $S^2\setminus\{\text{24 points}\}$ in the Gromov-Hausdorff sense; the metric is moreover a ``Monge-Amp\`ere metric" in the sense that its local convex potentials solve a real Monge-Amp\`ere equation.

This motivates the idea of first constructing a Monge-Amp\`ere metric on the base, with some number of isolated singularities, and then using that to study the degenerating complex structure. In dimension 2, the first examples of such metrics were constructed in \cite{GSVY90}; subsequently, Loftin \cite{Lof05} constructed many such metrics on $S^2$ and showed they were asymptotic to the metric of Gross-Wilson near the singular points. One important property that he is unable to show however is the integrality of the induced singular affine structure. 

On the other hand, if one assumes the existence of such a Monge-Amp\`ere metric (in some form), then Yang Li has recently shown in a series of breakthrough works \cite{Li22a, Li20} how to relate it to the degenerating complex metrics, thereby proving a version of the SYZ conjecture on certain degenerating families.

It is thus an important question to study these singular Monge-Amp\`ere metrics in greater detail. This is a highly non-trivial question, see e.g. \cite{Li20, Li22a} for a discussion of the difficulties. In certain cases however, progress can be made -- we focus in particular on the recent paper \cite{HJMM22}, where the problem is solved satisfactorily by assuming a high degree of symmetry on the base. In that paper, by realizing $S^2$ as the boundary of a symmetric $3$-simplex, $\p\Delta$, one obtains a natural singular affine structure; let $\lambda$ be the corresponding Lebesgue measure (coming from pulling back the usual Lebesgue measure in charts). A notion of ``$c$-convexity" for continuous functions on $\p\Delta$ is then defined, and it is shown there exists an (essentially) unique $c$-convex solution to the ``tropical Monge-Amp\`ere equation:"
\begin{equation}\label{asdf}
\nu_\psi = \lambda,
\end{equation}
where $\nu_\psi$ is the ``tropical Monge-Amp\`ere measure," see \cite{HJMM22}. The Hessian metric, $g$, derived from $\psi$ is then shown to be the required Monge-Amp\`ere metric; it has singularities at exactly 6 points, corresponding to the midpoints of the edges of $\p\Delta$.

The metric $g$ can be plugged into the machinery of \cite{Li20} for certain degenerating families of K3's, and so it is natural to relate it to the metrics of Gross-Wilson and Loftin. This is the task we take up in this paper:

\begin{thm}\label{ThrmA}
Suppose that $g$ is the Monge-Amp\`ere metric on $S^2$ minus 6 points constructed in \cite{HJMM22}. Then, near any singular point, there exists a complex coordinate $z$ such that $g(z)$ is asymptotic to the Gross-Wilson metric \cite[Section 3]{Lof05}, i.e.:
\[
g(z) = (-\log|z| + O(1)) |dz|^2.
\]
\end{thm}

The main difficulty in showing Theorem \ref{ThrmA} is in producing the complex coordinate; \cite{GW00, Lof05} work exclusively in complex coordinates, which is convenient for describing the asymptotics of the metric, but provides very little explicit information about the affine structure and convex potential. Meanwhile, \cite{HJMM22} starts by fixing the affine structure (so that it is automatically integral, for instance); this comes at the cost of locking the complex structure behind a non-explicit coordinate change.

We show that, using a local convex potential $\vp$ for $g$, it is possible to compute the requisite coordinate transformation in terms of the derivatives of $\vp$. A technical difficulty is that $\vp$ cannot be chosen to be smooth in a neighborhood of the singular point; this is a natural phenomenon, but makes it necessary to chose the coordinate carefully, so that it is continuous. This is accomplished by using symmetries. We then show that the identity function is a ``special" holomorphic coordinate for the special K\"ahler structure induced by $g$. This utilizes computations from \cite{CH19, Hay15}, and ultimately shows the conformal factor is harmonic, from whence Theorem \ref{ThrmA} follows.

In fact, our method seems to work more generally for Monge-Amp\`ere metrics coming from optimal transport problems where there is a joint $\Z_2\times \Z_2$ symmetry on the domain and range. While we do not pursue this further, it seems an interesting problem to look for more general transport problems where similar results hold; such problems may be relevant to the metrics contructed in \cite{AH23, Li23}, for instance.

Finally, the asymptotics we deduce using our conformal transformation can be used to show that the potential of the metric in \cite{HJMM22} is not $C^{1,1}$ across the singular points:

\begin{thm}\label{ThrmB}
Let $\psi$ be the $c$-convex solution to \eqref{asdf} on $\p\Delta$. Then $\psi$ is not $C^{1,1}$ across any of the 6 singular points.
\end{thm}

By the main result of \cite{SY20}, $\psi$ is $C^{1,\alpha}$ for any $0 < \alpha < 1$, so Theorem \ref{ThrmB} is optimal.

We conclude the introduction with a brief outline of the rest of the paper. In Section 2, we recall the setup of \cite{HJMM22}, and use it to derive a local optimal transport problem, whose solution can be used to recover the $c$-convex function $\psi$. In Section 3, we investigate the symmetries of the optimal transport problem, and use them to show that certain functions, constructed from the derivatives of the convex potential $\vp$, are continuous. Finally, in Section 4, we construct our isothermal coordinates, and prove Theorems \ref{ThrmA} and \ref{ThrmB}.

\subsection{Acknowledgements} We would like to thank Jakob Hultgren for interesting comments. This work came out of an REU project at the University of Michigan in the summer of 2023, funded by grants DMS-1900025 and DMS-2154380 from the NSF.

\section{The Tropical Monge-Amp\`ere Equation in Dimension 2} \label{section: Setup}

In this section, we recall the geometric setup in \cite{HJMM22} and reduce the (symmetric) tropical Monge-Amp\`ere equation to a planar optimal transport problem. This section is mainly motivational; excepting Corollary \ref{log type} and Theorems \ref{ThrmA} and \ref{ThrmB}, the rest of the paper is independent of the results mentioned here.

\subsection{The Tropical Monge-Amp\`ere Measure}

In \cite{HJMM22}, a geometric version of the Legendre transform, called the $c$-transform, was defined and studied for functions defined on the boundary of a simplex in $\R^{d+1}$. This $c$-transform then naturally gave rise to a Monge-Amp\`ere type operator, which was shown to agree with the Alexandrov Monge-Amp\`ere operator in charts (in a suitable sense). 

In this section, we shall recall the exact setup in \cite{HJMM22} when the dimension is $2$. Our simplex, $\Delta$, is defined to be the hull of the points:
\begin{align*}
    m_0 &= (1,1,1) & m_1 &= (-3,1,1) \\
    m_2 &= (1,-3,1) & m_3 &= (1,1,-3)
\end{align*}
inside $\R^3$. We write $A := \p \Delta$. We also consider the dual simplex, $\Delta^\vee \subseteq (\bbR^3)^*$:
\[
\Delta^\vee = \left\{n\in(\bbR^3)^*\ |\  \sup_{m\in\Delta}\langle m, n\rangle = \max_{0\leq i \leq 3}\langle m_i, n \rangle \leq 1\right\}.
\]

\begin{center}
\begin{figure}
\begin{tikzpicture}[scale=0.6]
\coordinate (0) at (2,-2);
\coordinate (1) at (3,2);
\coordinate (2) at (0,4);
\coordinate (3) at (-3,0);

\node[right] at (1) {$n_2$};
\node[below] at (0) {$n_1$};
\node[above] at (2) {$n_0$};
\node[left] at (3) {$n_3$};

\coordinate (01) at (2.5,0);
\coordinate (02) at (1,1);
\coordinate (03) at (-0.5,-1);
\coordinate (12) at (1.5,3);
\coordinate (23) at (-1.5,2);

\coordinate (012) at (5/3,4/3);
\coordinate (023) at (-1/3,2/3);
\coordinate (123) at (0,2);

\filldraw (0) circle (1pt);
\filldraw (1) circle (1pt);
\filldraw (2) circle (1pt);
\filldraw (3) circle (1pt);
\filldraw (01) circle (2pt);
\filldraw (02) circle (2pt);
\filldraw (03) circle (2pt);
\filldraw (12) circle (2pt);
\filldraw (23) circle (2pt);

\fill[blue, fill opacity=0.2] (2) -- (23) -- (023)--(02)--(012)--(12);
\draw (23)--(023)--(02)--(012)--(12);
\draw[dashed] (12)--(123)--(23);
\draw (0)--(1)--(2)--(3)--(0)--(2);
\draw[dashed] (1)--(3);
\end{tikzpicture}
\caption{The region $T_0\subset B$}
\label{fig:T_0}
\end{figure}
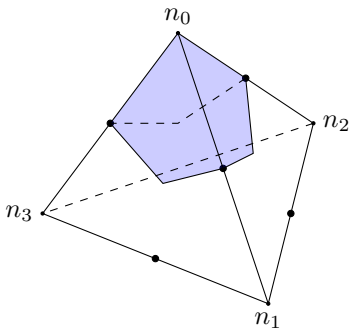
\end{center}

We write $B := \p\Delta^\vee$. One may check that $\Delta^\vee$ is the convex hull of the following points:
\begin{align*}
    n_0 &= (-1,-1,-1) & n_1 &= (1,0,0) \\
    n_2 &= (0,1,0) & n_3 &= (0,0,1)
\end{align*}

Write $m_{ij} := \frac{m_i + m_j}{2}$ and $n_{ij} := \frac{n_i + n_j}{2}$, $i \not= j$, for the midpoints of the edges of $A$ and $B$. Set:
\[
A_{\sing} := \cup_{i\not= j} \{m_{ij}\},\quad B_{\sing} := \cup_{i\not= j} \{n_{ij}\}, \quad A_{\reg} = A\setminus A_{\sing},\text{ and }B_{\reg} := B\setminus B_{\sing}.
\]
We can now define flat atlases on $A_{\reg}$ and $B_{\reg}$, which turn $A$ and $B$ into singular integral affine manifolds. These atlases have two different types of charts; the first are the open faces of $A$ and $B$. We shall write $\sigma_i\subset A$ for the (closed) face of $A$ which does not include $m_i$ as a vertex, i.e. $\sigma_i := \mathrm{Hull}\{m_j \ |\  j\neq i\}$. We define $\tau_i\subset B$ similarly. We may use any volume preserving linear projection onto an open subset of $\R^2$ as a coordinate function on $\sigma_i^\circ$ (similarly for $\tau_i^\circ$).

The other charts are centered on the vertices of the simplex. Define
\[
S_i = \{m\in A: \langle m, n_i \rangle = \min_j \langle m, n_j \rangle\}\subset A;
\]
$T_i \subset B$ is defined similarly. Fix $i, j\in \{0, 1,2,3\}$, $j\not= i$, and let $k < \ell$ be the remaining two indices. Then we declare:
\[
p_{i,j}^{-1}(m) := (\langle m, n_j - n_k \rangle, \langle m, n_j - n_\ell \rangle)
\]
to be a coordinate function on $S^\circ_i$ (note however that $p_{i,j}^{-1}$ is defined on the larger set $\mathrm{Star}(m_i) := A\setminus \sigma_i$). If we write $p_{i,j}$ for the inverse of $p_{i,j}^{-1}$, then one can check that $p_{i,j}$ is a piecewise affine map on its domain, which is affine when restricted to $p_{i,j}^{-1}(\tau_k)$ for any $k\not= i$. It follows that the transition functions of this atlas are affine, making $A_{\reg}$ into an affine manifold \cite{KS06, HJMM22}.

On $T_i^\circ$, we similarly define coordinate functions by:
\[
q_{i,j}^{-1}(n) := \frac{1}{4}(\langle m_k - m_j, n \rangle, \langle m_\ell - m_j, n \rangle),
\]
where the factor of $-1/4$ is to preserve the volume and compatibility with the coordinate pairing between $A$ and $B$. Again, this turns $B_{\reg}$ into an integral affine manifold.

The integral affine structure furnishes $A_{\reg}$ with a natural measure, defined by pulling back the Lebesgue measure on $\R^2$. We write $\mu$ for the zero extension of this measures to $A$, and refer to it as the Lebesgue measure on $A$. We write $\nu$ for the Lebesgue measure on $B$.\\

We now recall the $c$-transform. Denote by $L^\infty(A)$ and $L^\infty(B)$ the spaces of bounded real valued functions on $A$ and $B$, respectively. Given $\vp\in L^\infty(A), \psi\in L^\infty(B)$, we may define their $c$-transforms as follows:
\[
L^\infty(B) \ni \vp^c := \sup_{m\in A}\langle m,n \rangle - \vp(m)
\]
\[
L^\infty(A) \ni \psi^c := \sup_{n\in B}\langle m,n \rangle - \psi(n)
\]
Note that $\vp^c, \psi^c$ are bounded. 
We say that $\vp\in L^\infty(A)$ is {\bf $c$-convex} if $\vp^{cc} = \vp$, and similarly for functions $\psi\in L^\infty(B)$.

Given a $c$-convex function $\psi\in L^\infty(B)$, we define its $c$-subgradient as the multivalued map $\p ^c\psi : B \to A$ given by
\[
(\p^c \psi)(n) := \{m\in A\ |\ \psi(n) + \psi^c(m) = \langle m, n \rangle\}.
\]
As in the case of the usual subgradient of a convex function, one may check, e.g. \cite[Proposition 4.10]{HJMM22}, that the pushforward $(\p^c\psi^c)_*\mu$ is a well-defined measure. We shall refer to this measure as the tropical Monge-Amp\`ere measure of $\psi$, and denote it by $\nu_\psi$.

%

\subsection{Reduction to Planar Optimal Transport}

In \cite{HJMM22}, it was shown that there exists a $c$-convex function $\psi$ on $B$, unique up to an additive constant, solving the tropical Monge-Amp\`ere equation:
\begin{equation}\label{cMA}
\nu_\psi = \nu.
\end{equation}
The function $\psi$ is symmetric, in the following sense: any permutation of the vertices of $B$ induces a natural affine self-map of $B$, which preserves the measure $\nu$. By uniqueness, one finds that $\psi$ is also preserved under pull-back by these maps.

Write $m_j$ for the function $n \mapsto \langle m_j, n\rangle$ on $B$. Given $i\not= j$, define the function:
\[
\psi_{i,j} := (\psi - m_j)\circ q_{i,j} \text{ on }q_{i,j}^{-1}(\mathrm{Star}(n_i)).
\]
It was shown in \cite[Proposition 3.26]{Li22a}, \cite[Lemma 3.13]{HJMM22} that $\psi_{i,j}$ is convex on $q_{i,j}^{-1}(\mathrm{Star}(n_i))$; moreover, when we restrict to the image of an (appropriate) coordinate chart, we have an explicit bijection between the $c$-subgradients of $\psi$ and the subgradients of $\psi_{i,j}$:
\begin{lem}\label{lem:stuff}\cite[Lemma~4.4]{HJMM22}
Suppose that $\psi$ is a bounded and symmetric real-valued $c$-convex function on $B$. If $i\ne j$, then $p_{j,i}^{-1}$ gives a bijection of $(\p^c\psi)(n)$ onto $\p\psi_{i,j}(q_{i,j}^{-1}(n))$ for any $n\in T_i^{\circ}$. The same result holds for any $n\in \tau_j^{\circ}$.
\end{lem}

The goal of this subsection is to use Lemma \ref{lem:stuff} and the symmetries of $\psi$ to find a planar optimal transportation problem which can recover the solution to \eqref{cMA} on a neighborhood of a point $n_{ij}\in B_{\sing}$. 

From now on, we fix $\psi$ to be a solution to \eqref{cMA}. Write $m_{ijk} = \frac{m_i + m_j + m_k}{3}$ and $n_{ijk} = \frac{n_i + n_j + n_k}{3}$, $i\not=j\not=k\not=i$, for the barycenters of the open faces of $A$ and $B$. 

We define a neighborhood $Q$ of the singular point $n_{01}$ by:
\[
Q := B\cap\text{Hull}\{n_0, n_{013}, n_1, n_{012}\};
\]
using the proof of \cite[Lemma~4.1]{HJMM22}, one can check that $\p^c\psi$ maps $Q$ onto the region:
\[
P := A\cap\text{Hull}\{m_2, m_{023}, m_3, m_{123}\}.
\]
Briefly, to see this, subdivide $P$ and $Q$ into four triangular subregions; see Figure \ref{fig:kite regions}. Then, check that the $c$-pairing is maximized for points between corresponding subregions, by mimicking the proof of \cite[Lemma~2.1]{HJMM22}. Thus, for symmetric $c$-convex functions, the $c$-subgradient must map each subregion to its maximizing pair.
\begin{center}
\begin{figure}
\begin{tikzpicture}[scale=0.6]

\coordinate (d0) at (10,-2);
\coordinate (d1) at (11,2);
\coordinate (d2) at (8,4);
\coordinate (d3) at (5,0);

\node[right] at (d1) {$m_2$};
\node[below] at (d0) {$m_1$};
\node[above] at (d2) {$m_0$};
\node[left] at (d3) {$m_3$};

\coordinate (d01) at (10.5,0);
\coordinate (d02) at (9,1);
\coordinate (d03) at (7.5,-1);
\coordinate (d12) at (9.5,3);
\coordinate (d23) at (6.5,2);
\coordinate (d13) at (8,1);

\coordinate (d012) at (29/3,4/3);
\coordinate (d023) at (23/3,2/3);
\coordinate (d123) at (8,2);
\coordinate (d013) at (26/3,0);

\filldraw (d0) circle (1pt);
\filldraw (d1) circle (1pt);
\filldraw (d2) circle (1pt);
\filldraw (d3) circle (1pt);
\filldraw (d01) circle (2pt);
\filldraw (d02) circle (2pt);
\filldraw (d03) circle (2pt);
\filldraw (d12) circle (2pt);
\filldraw (d23) circle (2pt);

\fill[yellow, fill opacity=0.2] (d013)--(d1)--(d123)--(d3);
\draw[dashed] (d013)--(d1)--(d123)--(d3)--(d013);
\draw[dashed] (d013)--(d13)--(d123);
\draw (d0)--(d1)--(d2)--(d3)--(d0)--(d2);
\draw[dashed] (d1)--(d3);

\coordinate (0) at (2,-2);
\coordinate (1) at (3,2);
\coordinate (2) at (0,4);
\coordinate (3) at (-3,0);

\node[right] at (1) {$n_2$};
\node[below] at (0) {$n_1$};
\node[above] at (2) {$n_0$};
\node[left] at (3) {$n_3$};

\coordinate (01) at (2.5,0);
\coordinate (02) at (1,1);
\coordinate (03) at (-0.5,-1);
\coordinate (12) at (1.5,3);
\coordinate (23) at (-1.5,2);

\coordinate (012) at (5/3,4/3);
\coordinate (023) at (-1/3,2/3);
\coordinate (123) at (0,2);

\filldraw (0) circle (1pt);
\filldraw (1) circle (1pt);
\filldraw (2) circle (1pt);
\filldraw (3) circle (1pt);
\filldraw (01) circle (2pt);
\filldraw (02) circle (2pt);
\filldraw (03) circle (2pt);
\filldraw (12) circle (2pt);
\filldraw (23) circle (2pt);

\fill[blue, fill opacity=0.2] (2)--(023)--(0)--(012);
\draw (2)--(023)--(0)--(012)--(2);

\draw (023)--(02)--(012); 
\draw (0)--(1)--(2)--(3)--(0)--(2);
\draw[dashed] (1)--(3);
\end{tikzpicture}
\caption{The regions $Q\subset B$ and $P\subset A$ and their four subregions.}
\label{fig:kite regions}
\end{figure}
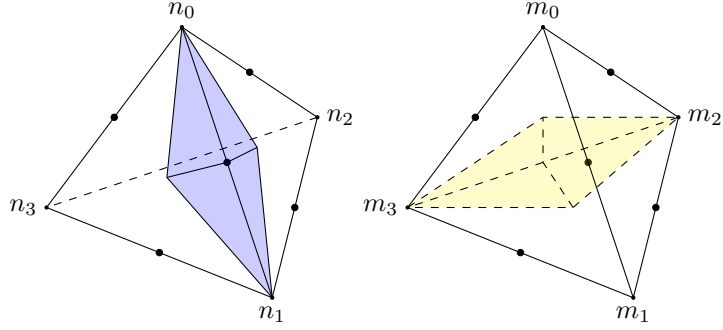
\end{center}

Now, a brief computation shows that the region $Q$ is mapped by $q_{0,1}^{-1}$ onto the kite-like region:
\[
\Omega := \mathrm{Hull}\{ (0,0), (\tfrac{1}{3}, 0), (1,1), (0, \tfrac{1}{3})\}  \subset \R^2;
\]
we have $q_{0,1}^{-1}(n_{01}) = (\tfrac{1}{2},\tfrac{1}{2})$. We now identify the image of the slit kite $\Omega \setminus \{y = x \geq \tfrac{1}{2}\}$ under $\p\psi_{0,1}$, the subgradient map of $\psi_{0,1}$. We will do this by dividing the region $Q$ into three components and using \cite[Lemma 4.1]{HJMM22} to determine the image of these regions under $\p^c\psi$. Then, we use Lemma \ref{lem:stuff} to relate this to $\p \psi_{0,1}$; this will require us to change coordinates to deal with the parts of $Q$ which lie outside of $T_0$.

We start by noting that $\p^c\psi(n_{01}) = \{m_{23}\}$. Now consider the region $B\cap \mathrm{Hull}\{n_0, n_{013}, n_{01}, n_{012}\}\subset T_0\cap Q$; the image of this under $q_{0,1}^{-1}$ is given by
\[
\mathrm{Hull}\left\{(0,0), \left(\tfrac{1}{3},0\right), \left(\tfrac{1}{2},\tfrac{1}{2}\right), \left(0,\tfrac{1}{3}\right)\right\}.
\]
By Lemma \ref{lem:stuff}, the image of this region under $\p\psi_{0,1}$ is the triangular region $\mathrm{Hull}\{(\tfrac{4}{3},\tfrac{4}{3}), (4,0), (0,4)\}$; note that the singular point $(\tfrac{1}{2}, \tfrac{1}{2})$ is sent to the point $(2,2)$.

We now consider the triangular region $B\cap \mathrm{Hull}\{n_{013},n_1,n_{01}\} \subset Q\cap \tau_3$; the image of this region under $q_{0,1}^{-1}$ is the triangle $K := \mathrm{Hull}\{(\tfrac{1}{3}, 0), (1,1) (\tfrac{1}{2},\tfrac{1}{2})\}$.


In this case, Lemma \ref{lem:stuff} only applies to $\psi_{0,2}$; hence, we need to compute the relation between $\psi_{0,1}$ and $\psi_{0,2}$. Recall that:
\begin{align*}
\psi_{0,1} = (\psi - m_1)\circ q_{0,1} && \psi_{0,2} = (\psi - m_2)\circ q_{0,2}
\end{align*}
Thus:
\[
\psi_{0,1} \circ q_{0,1}^{-1}(n) + \langle m_1, n \rangle = \psi_{0,2} \circ q_{0,2}^{-1}(n) + \langle m_2, n \rangle
\]
\[
\psi_{0,1} \circ q_{0,1}^{-1}(n) = \psi_{0,2} \circ q_{0,2}^{-1}(n) + \langle m_2 - m_1, n \rangle.
\]
Writing $n^\prime := q_{01}^{-1}(n)$, we now have the formula:
\[
\psi_{0,1}(n^\prime) = \psi_{0,2} \circ q_{0,2}^{-1} \circ q_{0,1}(n^\prime) + \langle m_2 - m_1, q_{0,1}(n^\prime) \rangle
\]
A brief computation shows that
\[
q_{02}^{-1} \circ q_{01} = \begin{bmatrix}-1 & 0 \\ -1 & 1\end{bmatrix}
\]
on the region $K$. Moreover, we have:
\[
\langle m_1 - m_2, q_{01}(n^\prime) \rangle = \begin{bmatrix}1&0 \end{bmatrix}n^\prime,
\]
So by the chain rule, we have
\[
\p\psi_{0,1} = \begin{bmatrix}-1 & -1 \\ 0 & 1\end{bmatrix}\p\psi_{0,2} + \begin{bmatrix}1\\0 \end{bmatrix}
\]
on $K$. We may now apply Lemma \ref{lem:stuff} to see that $\p\psi_{0,1}(K)$ is the triangle
\[
\mathrm{Hull}\{(4,0), (\tfrac{16}{3}, 0), (2,2)\}.
\]

A similar computation can be performed for the region bounded by $n_{013},n_{1},n_{01}$; the resulting image under $\p\psi_{0,1}$ is the triangle with vertices $(0, 4), (0, \tfrac{16}{3})$, and $(2, 2)$. 

We have thus shown:
\begin{prop}\label{transport problem}

Define the planar regions:
\begin{align}
\Omega &:= \mathrm{Hull}\{ (0,0), (\tfrac{1}{3}, 0), (1,1), (0, \tfrac{1}{3})\} \label{ugh1} \\
\Theta &:= \mathrm{Hull}\{(\tfrac{4}{3},\tfrac{4}{3}), (4,0), (\tfrac{16}{3},0), (2, 2)\}\cup \mathrm{Hull}\{(\tfrac{4}{3},\tfrac{4}{3}), (0,4), (0,\tfrac{16}{3}), (2, 2)\}.\label{ugh2}
\end{align}
If $\psi$ is a $c$-convex solution to \eqref{cMA}, then $\psi_{0,1}$ is a convex function on $\Omega$ such that:
\[
\p\psi_{0,1}: \Omega\setminus\{y = x \geq 1/2\} \rightarrow \Theta.
\]

\end{prop}
The situation is depicted in Figure \ref{fig:domains}.

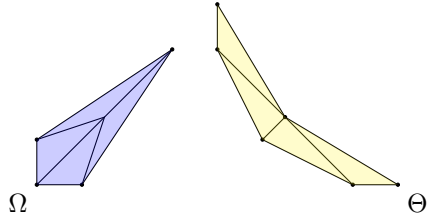
\begin{figure}
\begin{center}
\begin{tikzpicture} [scale=.6]

\coordinate (0) at (0-4,0);
\coordinate (1) at (0-4,1);
\coordinate (2) at (3/2-4,3/2);
\coordinate (3) at (1-4,0);
\coordinate (4) at (3-4,3);

\filldraw (0) circle (1pt);
\filldraw (1) circle (1pt);
\filldraw (3) circle (1pt);
\filldraw (4) circle (1pt);

\node[below left] at (0) {$\Omega$};

\draw (0)--(1)--(4)--(3)--(0);
\draw (0)--(4);
\draw (1)--(2)--(3);

\coordinate (5) at (1,1);
\coordinate (6) at (0,3);
\coordinate (7) at (0,4);
\coordinate (8) at (3/2,3/2);
\coordinate (9) at (4,0);
\coordinate (10) at (3,0);

\node[below right] at (9) {$\Theta$};

\filldraw (5) circle (1pt);
\filldraw (6) circle (1pt);
\filldraw (7) circle (1pt);
\filldraw (8) circle (1pt);
\filldraw (9) circle (1pt);
\filldraw (10) circle (1pt);

\draw (5)--(6)--(7)--(8)--(9)--(10)--(5);
\draw (5)--(8);
\draw (6)--(10);

\fill[blue, fill opacity=0.2] (0)--(1)--(4)--(3);
\fill[yellow, fill opacity=0.2] (5)--(6)--(7)--(8)--(9)--(10);
\end{tikzpicture}
\end{center}
\caption{The transport map $\p\psi_{0,1}$ maps $\Omega\setminus\{y=x \geq 1/2\}$ onto $\Theta$. Note that the domains have been scaled here to have equal area, but in actuality $\mathrm{Area}(\Theta)=16\cdot \mathrm{Area}(\Omega).$  }
\label{fig:domains}
\end{figure}
\section{Symmetries of the Planar Domains and the Convex Potential}

Suppose now that $\Omega$ and $\Theta$ are the domains given in \eqref{ugh1} and \eqref{ugh2}, respectively, and depicted in Figure \ref{fig:domains}. It will be convenient to rescale $\Theta$ by a factor of $1/4$, so that it has equal volume to $\Omega$. Let $\lambda$ be Lebesgue measure on $\R^2$.

For the rest of the paper, excepting Corollary \ref{log type}, we shall study the optimal transport problem from $\Omega$ to $\Theta$, without referring to the tropical Monge-Amp\`ere equation. By Brenier's Theorem \cite{Bre87} there exists an optimal transport map $T$ (with respect to the $L^1$-cost function) sending $\chi_{\Omega}\lambda$ to $\chi_{\Theta}\lambda$, and a convex potential $\vp$, unique up to an additive constant, such that $T = \p \vp$ and:
\[
\mathrm{MA}(\vp) = 1 \text{ on }\Omega,
\]
in the sense of Brenier. 

The regions $\Omega,\Theta$ have a combined $\Z_2\times \Z_2$-symmetry group by (piecewise) affine maps; if we shift both regions by $(-1/2, -1/2)$, the symmetries become (piecewise) linear, as the fixed point (the singularity) is shifted to the origin. This gives us the new regions (which we continue to denote by $\Omega, \Theta$):
\begin{align*}
\Omega &:= \mathrm{Hull}\{ (-\tfrac{1}{2},-\tfrac{1}{2}), (-\tfrac{1}{6}, -\tfrac{1}{2}), (\tfrac{1}{2},\tfrac{1}{2}), (-\tfrac{1}{2}, -\tfrac{1}{6})\}\\
\Theta &:= \mathrm{Hull}\{(-\tfrac{1}{6},-\tfrac{1}{6}), (\tfrac{1}{2},-\tfrac{1}{2}), (\tfrac{5}{6},-\tfrac{1}{2}), (0,0)\} \cup  \mathrm{Hull}\{(-\tfrac{1}{6},-\tfrac{1}{6}), (-\tfrac{1}{2},\tfrac{1}{2}), (-\tfrac{1}{2},\tfrac{5}{6}), (0,0)\}.
\end{align*}
If we now set:
\[
R := \begin{pmatrix} 0 & 1\\ 1 & 0 \end{pmatrix}\quad\text{ and }\quad A := \begin{cases}\begin{spm} 2 & -3 \\ 1 & -2 \end{spm} &\text{ if } y\geq x\\ \begin{spm} -2 & 1 \\ -3 & 2 \end{spm} &\text{ if } y\leq x \end{cases} 
\]
then one may check that the (piecewise) linear maps determined by $R$ and $A$ both preserve $\Omega$. Note that $R$ exchanges $\Omega\cap\{y > x\}$ and $\Omega\cap\{y < x\}$ while $A$ preserves these halves. Moreover, the symmetries of $\Theta$ are given exactly by $R^T$ ($= R$ of course) and $A^T$. Note that $R^2 = A^2 = \mathrm{Id}_2$.

We now show that these symmetries extend to $\vp$, and use them to determine the behaviour of both $\vp_x + \vp_y$ and $\vp_{xx} + 2\vp_{xy} + \vp_{yy}$ across the diagonal. We define:
\[
\Omega^+ := \{ (x,y)\in \Omega\ |\ y > x\}\quad \Omega^- := \{ (x,y)\in \Omega\ |\ y < x\},
\]
and
\begin{align*}
&\Omega^{++} := \{ (x,y)\in \Omega^+\ |\ 3y > x\}\quad \Omega^{+-} := \{ (x,y)\in \Omega^+\ |\ 3y < x\}\\
&\Omega^{-+} := \{ (x,y)\in \Omega^-\ |\ 3x > y \}\quad \Omega^{--} := \{ (x,y)\in \Omega^-\ |\ 3x < y\}
\end{align*}
Note that $R$ sends $\Omega^{+\pm}$ to $\Omega^{-\pm}$ and $A$ sends $\Omega^{\pm+}$ to $\Omega^{\pm-}$. We also split $\Theta$ into the regions:
\begin{align*}
&\Theta^{+} := \{ (x,y)\in \Theta\ |\ y > x \}\quad\quad \quad \Theta^- := \{ (x,y)\in \Theta\ |\ y < x\}\\
&\Theta^{++} := \{ (x,y)\in \Theta^+\ |\ y > -x\}\quad \Theta^{+-} := \{ (x,y)\in \Theta^+\ |\ y < -x\}\\
&\Theta^{-+} := \{ (x,y)\in \Theta^-\ |\ y > -x \}\quad \Theta^{--} := \{ (x,y)\in \Theta^-\ |\ y < -x\}.
\end{align*}
Again, we have $R(\Theta^{+\pm})=\Theta^{-\pm}$ and $A^T(\Theta^{\pm +}) = \Theta^{\pm -}$. These subdomains are depicted in Figure \ref{fig:subdomains}.

\begin{figure}
\begin{center}
\begin{tikzpicture} [scale=.6]

\coordinate (0) at (0-6,0);
\coordinate (1) at (0-6,1);
\coordinate (2) at (3/2-6,3/2);
\coordinate (3) at (1-6,0);
\coordinate (4) at (3-6,3);

\coordinate (01) at (0-6,1/2);
\coordinate (14) at (3/2-6,2);
\coordinate (34) at (2-6,3/2);
\coordinate (03) at (1/2-6,0);

\filldraw (0) circle (1pt);
\filldraw (1) circle (1pt);
\filldraw (2) circle (1pt);
\filldraw (3) circle (1pt);
\filldraw (4) circle (1pt);

\node[left] at (01) {$\Omega^{+-}$};
\node[above left] at (14) {$\Omega^{++}$};
\node[below right] at (34) {$\Omega^{-+}$};
\node[below] at (03) {$\Omega^{--}$};

\draw (0)--(1)--(4)--(3)--(0);
\draw (0)--(4);
\draw (1)--(2)--(3);

\coordinate (5) at (1,1);
\coordinate (6) at (0,3);
\coordinate (7) at (0,4);
\coordinate (8) at (3/2,3/2);
\coordinate (9) at (4,0);
\coordinate (10) at (3,0);

\coordinate (56) at (1/2,2);
\coordinate (78) at (3/4,11/4);
\coordinate (89) at (11/4,3/4);
\coordinate (510) at (2,1/2);

\node[left] at (56) {$\Theta^{+-}$};
\node[above right] at (78) {$\Theta^{++}$};
\node[above right] at (89) {$\Theta^{-+}$};
\node[below] at (510) {$\Theta^{--}$};

\filldraw (5) circle (1pt);
\filldraw (6) circle (1pt);
\filldraw (7) circle (1pt);
\filldraw (8) circle (1pt);
\filldraw (9) circle (1pt);
\filldraw (10) circle (1pt);

\draw (5)--(6)--(7)--(8)--(9)--(10)--(5);
\draw (5)--(8);
\draw (6)--(10);

\fill[blue, fill opacity=0.2] (0)--(1)--(4)--(3);
\fill[yellow, fill opacity=0.2] (5)--(6)--(7)--(8)--(9)--(10);
\end{tikzpicture}
\end{center}
\caption{Each domain admits four subdomains of equal area obtained by drawing edges between vertices and the barycenter.  It can be shown that the optimal map preserves this decomposition.}
\label{fig:subdomains}
\end{figure}
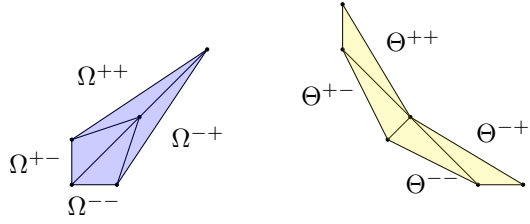
\begin{prop}
\label{prop: convexity and smoothness}

The potential $\vp$ is smooth on $\Omega \setminus \{ y = x \geq  0\}$. Moreover:
\[
\vp = \vp\circ R = \vp \circ A \text{ on }\Omega\setminus \{y = x\},
\] and $\p \vp (\Omega^{\pm\pm})=\Theta^{\pm \pm}.$ In addition, it is strictly convex on this domain, in the sense that $D^2(\vp)>0$ at all points in the interior.
\end{prop}
\begin{proof}

The proposition follows from the discussion in Section 2, but here we give a proof using only optimal transport theory. 

We work with $\psi := \vp^\vee$, the Legendre transform of $\vp$. The map $\p\psi$ is the optimal map sending $\Theta$ to $\Omega$ with $\Omega$ convex; thus, by \cite{Caf92}, $\psi$ is $C^{2}$ and strictly convex in $\Theta$, so $\p\psi$ is a local homeomorphism whose image is $\Omega$, in an a.e. sense. This regularity of $\psi$, along with some standard results of convex analysis \cite{Roc70}, imply $\psi^\vee=\vp$ and therefore $\p \vp \circ \p \psi  =\mathrm{Id}_2$. It will therefore suffice to show $\p\psi(\Theta^{\pm\pm})=\Omega^{\pm\pm},$ and then conclude
\[
\p \vp(\Omega^{\pm\pm})=\p \vp\circ \p \psi(\Theta^{\pm\pm})=\Theta^{\pm\pm}.
\]

The necessary fact is the \emph{cyclical monotonicity} of the transport map $\p \psi.$ That is, for points $v_1,\cdots, v_{n+1}\in \Theta,$ where $v_{n+1}=v_1,$ we have
\[\sum_{k=1}^n\langle v_{k+1},\p\psi(v_{k+1})-\p \psi(v_k)\rangle \geq 0.\] For $n=2,$ this reduces to 
\[\langle v-w,\p \psi( v) - \p\psi( w)\rangle \geq 0\] for any $v,w\in \Theta.$ 

Since the symmetry $R^T$ is linear and $R^T(\Theta)=\Theta$, the function $\psi \circ R^T$ is another convex function on $\Theta.$ Additionally, by the chain rule \[\p(\psi\circ R^T)(\Theta)=R\circ \p \psi\circ R^T(\Theta)=R\circ \p\psi(\Theta)=R(\Omega)=\Omega.\] Therefore by the uniqueness part of Brenier's Theorem, we have $\psi=\psi\circ R$. Now take a point $(x,y)\in \Theta.$ Cyclical monotonicity tells us
\[
\langle (x,y)-R^T(x,y),\p\psi(x,y)-\p(\psi\circ R^T)(R^T(x,y))\rangle \geq 0.
\]
Then by the chain rule we have
\[
\langle (x-y,y-x),(I-R)\p\psi(x,y) \rangle\geq 0 .
\]
Then the above inequality tells us
\[
(x-y)(\psi_x-\psi_y)+(y-x)(-\psi_x+\psi_y)=2(x-y)(\psi_x-\psi_y)\geq 0.
\]
Since $x<y$ on $\Theta^+$ and $y<x$ on $\Theta^-$, we find that $\psi_x\leq \psi_y$ on $\Theta^+$ and $\psi_y\leq \psi_x$ on $\Theta^{+}.$ This fact, coupled with the fact that $\p \psi$ is an open map preserving Lebesgue measure, implies $\p\psi(\Theta^{\pm})=\Omega^{\pm}$. In fact, the uniqueness part of Brenier's Theorem implies that $\p\psi$ is also the optimal map in the transport problems between $\Theta^{\pm}$ and $\Omega^\pm.$ 

We now demonstrate the further decomposition $\p\psi(\Theta^{\pm\pm})=\Omega^{\pm\pm}.$ We restrict our focus to $\Theta^-;$ the other half follows similarly. Since $A^T$ is linear on $\Theta^-$ and $\partial \psi$ is optimal between $\Theta^-$ and $\Omega^-,$ applying the same argument as before gives $\psi=\psi \circ A^T$. Let $(x,y)\in \Omega^{-}.$ Then by cyclical monotonicity
\[
\langle (x,y)-A^T(x,y),\p\psi(x,y)-\p (\psi\circ A^T)(A^T(x,y))\rangle \geq 0.
\]
Simplifying gives
\[
\langle (3x+3y,-x-y), (I-A)\p \psi(x,y)\rangle \geq 0 .
\] 
Thus, 
\[
(3x+3y)(3\psi_x-\psi_y)+(-x-y)(3\psi_x-\psi_y)=2(x+y)(3\psi_x-\psi_y)\geq 0.
\]
Since $x<-y$  on $\Theta^{-+}$ and $y<-x$ on $\Theta^{--}$, we see that $3\psi_x\leq \psi_y$ on $\Theta^{-+}$ and $\psi_y\leq 3\psi_y $ on $\Theta^{--}$. As before we find $\p \psi(\Theta^{-\pm})=\Omega^{-\pm}$ and $\p\psi$ is optimal between $\Theta^{-\pm}$ and $\Omega^{-\pm}$. Analogous results hold for the other half.

We have succeeded in demonstrating $\p\vp(\Omega^{\pm\pm})=\Theta^{\pm\pm}$ and now turn our focus towards establishing symmetry and regularity of $\vp$ on $\Omega\setminus \{y=x\geq 0\}.$ Since $\psi$ is invariant under the involutive transformations $R^T$ and $A^T$ we find that $\vp$ is invariant under the  transformations $R$ and $A$. Since $\Theta^+$ and $\Omega^+$ are convex regions and $\p\vp(\Omega^{+})=\Theta^+$, by \cite{Caf92} $\vp$ is also strictly convex and smooth on $\Omega^+$; similar statements hold for $\vp$ on $\Omega^{-}$. Since $\p\vp(\Omega^{+-}\cup\Omega^{--})=\Theta^{+-}\cup \Theta^{--}$ and these regions are convex, \cite{Caf92} again gives smoothness and strict convexity of $\vp$ on $\Omega^{+-}\cup\Omega^{--}$. Thus we have demonstrated that $\vp$ is smooth and strictly convex on an open cover of $\Omega\setminus\{y=x\geq 0\}$ and therefore conclude.
\end{proof}

\begin{cor}\label{f and rho}
Define:
\[
 v := \vp_x + \vp_y \quad \text{ and }\quad \rho := v_x + v_y\quad  \text{ on }\ \ \Omega\setminus \{y = x \geq 0\}.
\]
Then $v = v\circ R = -v\circ A$ and $\rho = \rho\circ R = \rho\circ A$ on $\Omega\setminus \{y = x\}$, and  hence $v$ and $\rho$ admit continuous extensions to $\Omega$ and $\Omega^*$, respectively, which we still denote by $v$ and $\rho$. Moreover, $\rho$ is strictly positive on $\Omega^*$.
\end{cor}
\begin{proof}
Working on $\Omega\setminus \{y = x\}$, we have that:
\[
\p \vp = \p (\vp\circ A) = A^T (\p \vp)\circ A.
\]
We may now sum the columns of $A^T$ to conclude that:
\[
\vp_x + \vp_y = - \vp_x(A) - \vp_y(A),
\]
i.e. $v = -v\circ A$. Similarly, we see that $v = v\circ R$, since the columns of $R^T$ sum to $1$. The exact same process works to show $\rho = \rho\circ R = \rho\circ A$. 

The symmetries $v=-v\circ A$ and $\rho = \rho \circ A$ now imply that $v$ and $\rho$ extend continuously to $\Omega^*.$ In addition, \cite[Theorem C3]{Caf96} and \cite{SY20} dictate that $\vp$ is  $C^{1,\alpha}$ on $\Omega\setminus \{y=x\geq 0\}$ for some $\alpha>0.$ Hence $v$ extends continuously to $\Omega.$

We are left to check $\rho > 0$. By the symmetry $\rho=\rho\circ A,$ we can restrict ourselves to the set where $\vp$ is smooth. The Monge-Amp\`ere equation $\vp_{xx}\vp_{yy}-\vp_{xy}^2=1$ and the fact that the square root function is strictly increasing together imply
\[
\vp_{xy}\geq -\sqrt{\vp_{xx}\vp_{yy}-1}>-\sqrt{\vp_{xx}\vp_{yy}}.
\]
Thus,
\[
\vp_{xx}+2\vp_{xy}+\vp_{yy}>\vp_{xx}-2\sqrt{\vp_{xx}\vp_{yy}}+\vp_{yy}=(\sqrt{\vp_{xx}}-\sqrt{\vp_{yy}})^2\geq 0.
\qedhere\]
\end{proof}

\section{Isothermal Coordinates and the Special K\"ahler Structure}
\label{section: SKS}

We now find isothermal coordinates for the Hessian metric $g := \nabla d\vp$ on $\Omega\setminus\{x = y \geq  0\}$, in terms of the potential. By Corollary \ref{f and rho}, both the coordinates and the conformal factor extend continuously, but not smoothly, across the slit. However, the key point is that both fail to be smooth in exactly the same way -- namely, by the transformation $A$. This will allow us to show in Proposition \ref{conf sym} that the conformal factor actually extends smoothly in the isothermal coordinates.

We can then define a special K\"ahler structure (SKS) on the punctured disk in these isothermal coordinates. The isothermals we choose are moreover seen to be ``special holomorphic" for this SKS, and so by \cite[Proposition 10.2.9]{CH19}, we see that the conformal factor is harmonic. It follows that the SKS is of logarithmic type with order 1; as previously mentioned, this means it is of the same type as the Gross-Wilson metric \cite{GW00, Lof05}

\begin{prop}
\label{prop:coord change}
Define:
\[
u(x, y) := x - y.
\]
Then the coordinate change $f(x,y)=(u(x,y),v(x,y))$ is a homeomorphism on $\Omega$, smooth on $\Omega\setminus \{y=x\geq 0\},$ mapping the subdomains $\Omega^{\pm \pm}$ into the four quadrants of the plane and fixing the origin. Moreover, the $(u,v)$-coordinates are isothermal for $g$ on $\Omega\setminus\{y = x \geq  0\}$; specifically, we have:
\[
g = \rho^{-1} |du^2 + dv^2|\text{ on }\Omega\setminus\{ y = x \geq 0\}.
\]
\end{prop}
\begin{proof}

By writing down the $g$-Laplacian in $(x,y)$-coordinates, it is easy to check that (many) linear combinations of $x, y, \vp_x,$ and $\vp_y$ define isothermal coordinates for $g$ on the set where it is smooth. Here however, we will prove the proposition by direct computations.

By Corollary \ref{f and rho}, we see that $f$ is continuous on $\Omega$ and smooth on $\Omega\setminus\{y=x\geq 0\}.$ To demonstrate the injectivity of $f$ it will suffice to prove it along lines of the form $x-y=c$. We have:
\[
\tfrac{d}{dx}v(x,x-c)=\vp_{xx}(x,x-c)+2\vp_{xy}(x,x-c)+\vp_{yy}(x,x-c),
\] 
on $(\Omega\setminus \{y=x\geq 0\})\cap\{x-y = c\}$, so $v$ is strictly increasing along these lines. We can extend this to the exceptional line $\{x = y\}$ by the symmetry $v=-v\circ A$ (which additionally implies $v(0,0) = 0$). Thus, $f$ is a continuous injection and therefore a homeomorphism.

To see that $f$ maps the subdomains $\Omega^{\pm\pm}$ to the four quadrants of the plane, we observe first that $u<0$ on $\Omega^+$ and $u>0$ on $\Omega^-.$ Then by Proposition \ref{prop: convexity and smoothness}, along with the definitions of $\Theta^{\pm\pm},$ we see that $v<0$ on $\Omega^{+-}\cup \Omega^{--} $ and $v>0$ on $\Omega^{++}\cup \Omega^{-+}.$

The proof that $(u,v)$-coordinates are isothermal is a computation. We have:
\begin{align*}
du^2 + dv^2 &= (dx-dy)^2+((\vp_{xx}+\vp_{xy})dx + (\vp_{xy}+\vp_{yy})dy)^2\\
&=(1+\vp_{xx}^2+2\vp_{xx}\vp_{xy}+\vp_{xy}^2)dx^2\\
&+(-1+\vp_{xx}\vp_{xy}+\vp_{xx}\vp_{yy}+\vp_{xy}^2+\vp_{xy}\vp_{yy})(dxdy+dydx)\\
&+(1+\vp_{xy}^2+2\vp_{xy}\vp_{yy}+\vp_{yy}^2)dy^2.
\end{align*}
Using the Monge-Amp\`ere equation $\vp_{xx}\vp_{yy}-\vp_{xy}^2=1$ to eliminate the constant terms then gives:
\begin{align*}
du^2 + dv^2 &=\vp_{xx}(\vp_{xx}+2\vp_{xy}+\vp_{yy})dx^2\\
&+\vp_{xy}(\vp_{xx}+2\vp_{xy}+\vp_{yy})(dxdy+dydx)\\
&+\vp_{yy}(\vp_{xx}+2\vp_{xy}+\vp_{yy})dy^2\\
&=(\vp_{xx}+2\vp_{xy}+\vp_{yy})g.\qedhere
\end{align*}
\end{proof}
\begin{prop}\label{conf sym}

Let $\ti{\Omega}=f(\Omega)$ and $\ti{\Omega}^* = \ti{\Omega}\setminus\{(0,0)\}$. Then $\ti{\Omega}$ is invariant under the linear transformations:
\[
G = \begin{pmatrix} -1  & 0\\ 0 & 1 \end{pmatrix} \quad \quad H =  \begin{pmatrix} 1  & 0\\ 0 & -1 \end{pmatrix}.
\]
Moreover, $\rho$ is invariant under these transformations, and hence $g$ admits a smooth extension
\[
g := \rho^{-1} |du^2 + dv^2|
\]
to all of $\ti{\Omega}^*$, which is K\"ahler with respect to the standard $(u,v)$-complex structure.
\end{prop}
\begin{proof}
The invariance of $\ti\Omega$ under the transformations $G$ and $H$ follows from the symmetries $u=-u\circ R=u\circ A$ and $v=v\circ R=-v\circ A$. Indeed, $f\circ R(x,y)=G\circ f(x,y)$ and $f\circ A(x,y)=H\circ f(x,y),$ so 
\begin{align*}
G(\ti\Omega)=G\circ f(\Omega)=f\circ R(\Omega)=f(\Omega)=\ti \Omega&\\ H(\ti\Omega)=H\circ f(\Omega)=f\circ A(\Omega)=f(\Omega)=\ti \Omega&.   
\end{align*}  

Additionally, Proposition \ref{prop:coord change} implies that $f$ is smooth on $\Omega\setminus\{y=x\geq 0\}$ and that $f(\Omega\setminus \{y=x\geq 0\})=\ti\Omega\setminus \{u=0,v\geq0\}$. It follows that $\rho$ is smooth on $\ti \Omega\setminus \{u=0,v\geq 0\}$. The symmetry $\rho=\rho\circ A$ in $(x,y)$-coordinates becomes the symmetry $\rho=\rho\circ H$ in $(u,v)$-coordinates, and since $H$ is linear, we find that $\rho$ is actually smooth on $\ti\Omega^*$; it follows that $g$ admits the desired smooth extension to all of $\ti{\Omega}^*$.

Finally, since $f$ is isothermal for $g$, we have that $g$ is compatible with the standard complex structure on $\ti{\Omega}^*$; it is thus trivially K\"ahler. 
\end{proof}

\begin{prop}

Let $\nabla$ denote the pull-back of the flat connection on $\Omega^*$ by $f$. Then $(g, \nabla)$ is a special K\"ahler structure on $\ti{\Omega}\setminus \{u=0,v\geq 0\}$. 

\end{prop}
\begin{proof}

For the duration of the proof we restrict ourselves to $\ti\Omega\setminus\{u=0,v\geq 0\}$. We need to show that $\nabla$ is a symplectic, torsion-free, flat connection satisfying the identity
\begin{equation}\label{special kaehler cond}
(\nabla_XJ)Y=(\nabla_Y J)X,
\end{equation}
where $J$ is the standard complex structure on $\ti{\Omega}$. That $\nabla$ is torsion-free and flat is immediate from the definition. To check that $\nabla$ is symplectic, we calculate the K\"ahler form in $(x,y)$-coordinates:
\[
\omega = g(J\cdot,\cdot) =\frac{1}{\vp_{xx}+2\vp_{xy}+\vp_{yy}}(-dvdu+dudv) =\frac{2\ du\wedge dv}{\vp_{xx}+2\vp_{xy}+\vp_{yy}}.
\]
Then:
\begin{align*}
\omega &=\frac{2(u_x\ dx + u_y \ dy)\wedge(v_x \ dx + v_y \ dy)}{\vp_{xx}+2\vp_{xy}+\vp_{yy}} =\frac{2(u_xv_y-u_yv_x)dx\wedge dy}{\vp_{xx}+2\vp_{xy}+\vp_{yy}} =2 \ dx\wedge dy.
\end{align*} 
Thus, $\omega$ is parallel with respect to $\nabla.$

Finally, we check \eqref{special kaehler cond}. Write $(x,y) = (x^1, x^2)$, and $J=J^j_k\, dx^k \otimes \frac{\p}{\p x^j}$
for the complex structure, using index notation. Then
\begin{align*}
    \nabla J &= \p_\ell J^j_{k}\ dx^\ell\otimes dx^k\otimes \del{x^j},
\end{align*}
so \eqref{special kaehler cond} becomes $\p_\ell J^j_{k} = \p_j J^\ell_{k}$, where we have been writing $\p_\ell f := \frac{\p f}{\p x^\ell}$.
To verify this, we compute the complex structure in $(x,y)$-coordinates:
\begin{align*}
J&=-\del u \otimes dv + \del v \otimes du\\
&=-\frac{1}{u_xv_y-u_yv_x}(v_y \ \del x - v_x \ \del y)\otimes (v_x\ dx + v_y \ dy) \\ &+ \frac{1}{u_xv_y-u_yv_x}(-u_y\  \del x + u_x \ \del y)\otimes (u_x \ dx + u_y \ dy)\\
&= \rho^{-1}[(-v_xv_y - u_xu_y)( \del{x}\otimes dx) + (-v_y^2-u_y^2)(\del{x}\otimes dy)\\
&+ (v_x^2 + u_x^2)(\del{y}\otimes dx) + (v_xv_y + u_xu_y)(\del{y}\otimes dy) ].
\end{align*}
Exactly as before, we may use the Monge-Amp\`ere equation to compute that:
\begin{align*}
&u_x^2 + v_x^2 = 1 + (\vp_{xx} + \vp_{xy})^2 = \vp_{xx}\, \rho\\
&u_y^2 + v_y^2 = 1 + (\vp_{xy} + \vp_{yy})^2 = \vp_{yy}\, \rho\\
&u_xu_y + v_xv_y = \vp_{xy}\,\rho,
\end{align*}
so we conclude that: 
\begin{align*}
J = -\vp_{xy}\ \del{x}\otimes dx + -\vp_{yy} \ \del{x}\otimes dy + \vp_{xx}\ \del{y}\otimes dx + \vp_{xy} \ \del{y}\otimes dy.
\end{align*}
Since $\del{y}(-\vp_{xy}) = \del{x}(-\vp_{yy})$ and $\del{y}\vp_{xx} = \del{x}\vp_{xy}$, we find (\ref{special kaehler cond}) is satisfied.  Thus, $\nabla$ induces a special K\"ahler structure on $\ti\Omega\setminus \{u=0,v\geq 0\}.$
\end{proof}

The crucial observation is that the SKS $(g, \nabla)$ always admits a special holomorphic coordinate; with our specific choice of isothermals, it will actually be the identity function. As shown in \cite[Corollary 10.2.10]{CH19}, the existence of such a coordinate places severe restrictions on the possible singularities of the SKS.

\begin{prop}
The function:
\[
z := x - y + i(\vp_x + \vp_y) = u + i v
\]
is a special holomorphic coordinate for $(g, \nabla)$ on $\ti\Omega\setminus \{u=0,v\geq 0\}$.
\end{prop}
\begin{proof}
Recall that a holomorphic coordinate in said to be {\it special}, if the real part is flat with respect to $\nabla$. Since we have taken $\nabla$ to be the flat connection in $(x,y)$-coordinates, the real part of $z$ is clearly flat, and hence $z$ is special. 
\end{proof}

We may now prove Theorem \ref{ThrmA}; using Corollary \ref{log type} below, we can arrange that $C=1$ by scaling $z$.

\begin{thm}
\label{thm:harmonic factor}
The conformal factor $\rho^{-1}$ is harmonic on $\ti{\Omega}^*$, and hence:
\[
\rho^{-1} = -C \log \abs{z} + h(z),
\]
for some $C \geq 0$ and some smooth harmonic function $h(z)$ on $\ti{\Omega}$.
\end{thm}
\begin{proof}
We follow \cite[Proposition 10.2.9]{CH19}. Write:
\[
\alpha = \begin{bmatrix} \alpha^1_1 & \alpha^1_2\\ \alpha_1^2 & \alpha_2^2 \end{bmatrix}
\]
for the connection $1$-form of $\nabla$ in $(u,v)$-coordinates on $\ti\Omega\setminus \{u=0,v\geq 0\}$. The coordinate $u$ is flat with respect to $\nabla$, so we have:
\[
0 = \nabla(d u) = -\alpha_1^1\otimes du - \alpha_2^1\otimes dv;
\]
it follows that $\alpha_1^1 = \alpha_2^1 = 0$. The symplectic form is $\omega := \rho^{-1} i du\wedge dv$; since $\nabla\omega = 0$ also, we get:
\[
0 = \nabla \omega = -\rho^{-1}\alpha_2^2 - \rho^{-2} d\rho \implies \alpha_2^2 = -d\log \rho.
\]
By the proof of \cite[Proposition 2.1]{Hay15}, we have $\alpha_1^2 = *\alpha_2^2$, where $*$ is the Hodge operator with respect to the standard metric $du^2 + dv^2$. Then, the fact that $\nabla$ is flat and $\alpha_1^1 = 0$ implies that:
\[
d^*\alpha_2^2 = -|\alpha_2^2|^2.
\]
We then have that:
\[
\Delta \log \rho = d^*d\log \rho = \abs{\alpha_2^2}^2 = \abs{d \rho}^2.
\]
It follows that:
\[
\Delta \rho^{-1} = \Delta e^{-\log \rho} = -\rho^{-1}\Delta \log \rho + \rho^{-1} \abs{d\log\rho}^2 = 0.
\]
Knowing $\Delta\rho^{-1}=0$ on $\ti\Omega\setminus \{u=0,v\geq 0\}$ in conjunction with knowing $\rho>0$ is smooth on $\ti\Omega^*$ lets us conclude $\Delta\rho^{-1}=0$ on all of $\ti\Omega^*.$

A well-known result for harmonic functions (e.g. \cite[Theorem 3.9]{ABR01}), now shows $\rho^{-1} = C\log |z| + h(z)$.
\end{proof}
\begin{cor}\label{log type}
In the classification of \cite{CH20}, the SKS $(g, \nabla)$ is of logarithmic type with order 1.
\end{cor}
\begin{proof}
The proof consists of ruling out the possibility $C=0$. To do this, we show that in such a case the connection $\nabla,$ a priori a flat connection on $B_{\reg},$ extends smoothly to all of $B$ once we have completed our atlas with conformal charts around each singular point. Then as $B$ has the topology of the sphere, we arrive at a contradiction.

We begin by augmenting our flat atlas with conformal charts around each singular point, as we have done so far around the point $n_{01}$. To do this, we define 
\[r_{i,j}^{-1}(n):=(x-y,(\psi_{i,j})_x+(\psi_{i,j})_y),\] where $q_{i,j}^{-1}(n)=(x,y),$ on the neighborhood of the point $n_{ij}$ enclosed by $n_{i},$ $n_{j},$ $n_{ijk},$ $n_{ij\ell},$ where $\{i,j,k,\ell\}=\{0,1,2,3\}.$ To see that this is homeomorphism to an open set of the plane, we repeat the proof of Proposition \ref{prop:coord change}. Note that around each singular point $n_{ij}$ we have two conformal charts $r_{i,j}^{-1}$ and $r_{j,i}^{-1}$.

We claim that $\{p_{i,j},r_{i,j}\}$ is a smooth atlas for $B.$ Inspecting the aforementioned regions over which $r_{i,j}^{-1}$ is defined, we see that it intersects with the conformal chart $r_{j,i}^{-1}$ and the flat charts \[q_{i,j}^{-1},q_{i,\ell}^{-1},q_{i,k}^{-1}, q_{j,i}^{-1}, q_{j,k}^{-1}, q_{j,\ell}^{-1}.\] First, the transition function from $r_{i,j}^{-1}$ to $r_{j,i}^{-1}$ is a reflection, which follows from the same analysis used to prove Proposition \ref{conf sym}. Then to prove compatibility of $r_{i,j}^{-1}$ with the aforementioned flat charts, it will suffice to check that the transition function from $q_{i,j}^{-1}$ to $r_{i,j}^{-1}$ is smooth, the rest can be had by further composition with already known smooth transition maps. To see this, we repeat the analysis used to prove Proposition \ref{prop:coord change}. Thus, we obtain a smooth atlas on $B.$

As demonstrated in the proof of Proposition \ref{thm:harmonic factor}, if $C=0,$ then the conformal factor $\rho$ becomes smooth as does the connection $\nabla.$ Then $\nabla$ is a flat connection on $S^2,$ enabling us to take a nonzero vector and complete it to a nonvanishing vector field on $S^2$ via parallel transport. This is a contradiction, and thus we conclude $C\not=0.$ 
\end{proof}

We can now show that the $C^2$-norm of $\vp$ blows up as we approach the origin. By combining this with Proposition \ref{transport problem}, we immediately arrive at Theorem \ref{ThrmB}.

\begin{thm}\label{ThrmNotB}
In $(x,y)$-coordinates, the second derivatives of $\vp$ blow up at the origin.
\end{thm}
\begin{proof}
We shall show that
\[
2\rho^{-1} \leq \vp_{xx} + \vp_{yy};
\]
since $\rho^{-1}$ blows up as $(x,y)\rightarrow (0,0)$, by Corollary \ref{log type}, the result follows.

We have $\vp_{xx}+2\vp_{xy}+\vp_{yy}=\rho,$ so:
\[
\vp_{xy}^2=\frac{\vp_{xx}^2+2\vp_{xx}\vp_{yy}+\vp_{yy}^2}{4} - \frac{\rho}{2}(\vp_{xx}+\vp_{yy})+\frac{\rho^2}{4}.
\] 
Substituting into the Monge-Amp\`ere equation gives:
\begin{align*}
  1 &= \vp_{xx}\vp_{yy}-\vp_{xy}^2\\
  &=-\tfrac{1}{4}\vp_{xx}^2+\tfrac{1}{2}\vp_{xx}\vp_{yy}-\tfrac{1}{4}\vp_{yy}^2+\frac{\rho}{2}(\vp_{xx} + \vp_{yy}) - \frac{\rho^2}{4}\\
  &=-\tfrac{1}{4}(\vp_{xx}-\vp_{yy})^2+\frac{\rho}{2}(\vp_{xx} + \vp_{yy}) - \frac{\rho^2}{4}.  
\end{align*}
Since $\rho > 0$, we conclude. As a remark, a further analysis of the above equation actually shows that each of $\vp_{xx}, \vp_{yy},$ and $-\vp_{xy}$ are bounded below by $(1-\varepsilon)\rho^{-1}$ as $(x,y)\rightarrow 0$, where $\varepsilon > 0$ is arbitrary.
%
\end{proof}


\end{document}